\documentclass[11pt,a4paper]{article}
\usepackage{etex} 
\usepackage{amsmath}
\usepackage{amssymb}
\usepackage{amsthm}
\usepackage{amsfonts}
\usepackage{enumerate}
\usepackage{pspicture}
\usepackage{pstricks}
\usepackage{pst-plot}
\usepackage{xspace} 
\usepackage{graphics} 
\usepackage{mathtools} 
\usepackage{graphicx}
\usepackage{pgfpages}
\usepackage{url}
\usepackage{booktabs}
\usepackage{verbatim}
\usepackage{hyperref}
\usepackage{breakurl}
\usepackage[margin=1.3in]{geometry}
\newcommand{\mna}[1]{{\mathcal{#1}}}

\newcommand{\onum}[1]{\mbox{$\overline{{#1}}$}} 
\newcommand{\unum}[1]{\mbox{$\underline{{#1}}$}}

\newcommand{\R}[0]{{\mathbb{R}}}

\def\eps{{\varepsilon}}
\newcommand{\mmid}[0]{;\,}		

\newcommand{\st}[0]{{\ \ \mbox{subject to}\ \ }}
	
\def\dmu{DMU${}_0$\xspace}

\def\nref#1{\mbox{(\ref{#1})}}

\newtheorem{theorem}{Theorem}

\newtheorem{proposition}{Proposition}

\theoremstyle{definition}

\newtheorem{example}{Example}

\begin{document}

\title{A novel data envelopment analysis ranking based on a robust approach}

\author{
  Milan Hlad\'{i}k\footnote{
Charles University, Faculty  of  Mathematics  and  Physics,
Department of Applied Mathematics, 
Malostransk\'e n\'am.~25, 11800, Prague, Czech Republic, 
e-mail: \texttt{milan.hladik@matfyz.cz}
}
}


\date{\today}
\maketitle

\begin{abstract}
We propose a novel DEA ranking based on a robust optimization viewpoint: the higher ranking for those DMU's that remain efficient even for larger simultaneous and independent variations of all data and vice versa. This ranking can be computed by solving generalized linear fractional programming problems, but we also present a tight linear programming approximation that preserves the order of rankings. We show many remarkable properties of our approach: It preserves the order of rankings compared to the classical approach, and it is unit invariant. It is naturally normalized, so it can be used as universal ranking of DMU's of unrelated models. It gives ranking not only for inefficient, but also for efficient decision making units. It can also be easily extended to generalized or alternative models, for instance to deal with interval data. We present several examples confirming the desirable properties of the method.
\end{abstract}


\section{Introduction}

Data envelopment analysis (DEA) \cite{CooSei2007,Zhu2016} is a method for evaluating the performance of a group of decision making units (DMU). The basic DEA model measures the DMU's such that it finds the most convenient weights of inputs and outputs factors such that the relative efficiency is maximal. Here, the relative efficiency is expressed as the weighted sum of outputs divided by the weighted sum of inputs .
The classical CCR model \cite{CharCoo1978} for ranking the decision making unit $0$ (denoted by \dmu) can be formulated as a linear program
\begin{align}\label{lpDEA}
\max\ y_0^Tu \st x_0^Tv \leq 1,\ Yu-Xv\leq0,\ u,v\geq0,
\end{align}
where 
\begin{itemize}\addtolength{\itemsep}{-0.3\baselineskip}
\item
$x_0\in\R^{n_1}$ is the input nonnegative vector for \dmu,
\item
$y_0\in\R^{n_2}$ is the output nonnegative vector for \dmu,
\item
$X\in\R^{m\times n_1}$ is the input nonnegative matrix for the other DMU's, in particular, the $i$th row of $X$ is the input vector for the $i$th DMU,
\item
$Y\in\R^{m\times n_2}$ is the output nonnegative matrix for the other DMU's, in particular, the $i$th row of $Y$ is the output vector for the $i$th DMU,
\item
$u$ and $v$ are vectors of variables representing output and input weights, respectively.
\end{itemize}

We will consider this model even though other models exist \cite{CooSei2007} as well as more economic compact formulations \cite{JahSol2008}. However, we will show (Section~\ref{ssBCC}) that our approach is easily extended to alternative models.

The aim of this paper is to bring a new efficiency ranking based on a robustness point of view. This ranking could be used not only for comparing individual DMU's (as the classical DEA method), but also for measuring stability and distances to efficiency or inefficiency.

There were already presented several robust optimization methods for DEA; see \cite{Alt2013,HafHaj2015,KavAbb2015,Lu2015,ShoHat2010,ShoSha2014,WanWei2010} and see also related sensitivity and stability developments in \cite{CooLi2001,SeiZhu1998}. 
Sensitivity analysis in DEA was primarily focused on perturbation of one DMU. In particular, a linear programming method to compute Chebyshev and 1-norm stability radius is proposed in \cite{CharHaa1992} for an additive model and in \cite{CharRou1996} for a ratio (CCR) model.

The above results mostly utilize robust optimization approach to deal with imprecise data, or study sensitivity w.r.t.\ data variations.
Our approach is substantially different -- we employ the robustness idea to propose a new ranking for DEA, which may or may not contain uncertain data, such that it measures relative efficiencies and their robustness in one. Similar idea was also discussed in \cite{RouSem1995b}, but data of only one DMU were considered. In contrast, we employ simultaneously all data together. Another ranking of DMU's based on robustness was proposed in \cite{RouSem1995}, but their motivation came up from game theory, and the ranking has another interpretation.

\section{Robust approach}

The underlying idea of the robust approach is to determine the largest allowable variations of all input and output data such that \dmu remains efficient (for efficient DMU's) or the smallest possible variation of the input and output data such that \dmu becomes efficient (for inefficient DMU's). The corresponding coefficient of variations gives us a new ranking based on a robustness viewpoint.  In other words, it can we viewed as a distance to the nearest inefficient point (for efficient DMU's) and vice verse. We choose a Chebyshev-like norm because it has a very simple and useful interpretation; this is analogous to the interpretation of the tolerance approach to sensitivity analysis in linear programming developed by Wendell \cite{Wen1985,Wen2004}.

Formally, define the new ranking $r$ as follows.
We use the relative $\delta$-neighborhood of the data
\begin{align*}
\mna{O}_{\delta}(x_0,y_0,X,Y)
=\{(x'_0,y'_0,X',Y')\mmid |x'_{ij}-x_{ij}|\leq\delta x_{ij},\ 
|y'_{ik}-y_{ik}|\leq\delta y_{ik},\ \forall i=0,1,\dots,m,\,\forall j,k\}.
\end{align*}
If \dmu is efficient, then its ranking is defined as $r=1+\delta^*$, where
\begin{align*}
\delta^*=\max\ \delta \st \mbox{\dmu is efficient for all data }
  (x'_0,y'_0,X',Y')\in\mna{O}_{\delta}(x_0,y_0,X,Y).
\end{align*}
If \dmu is inefficient, then its ranking is defined as $r=1+\delta^*$, where
\begin{align*}
\delta^*=-\min\ \delta \st \mbox{\dmu is efficient for some data }
  (x'_0,y'_0,X',Y')\in\mna{O}_{\delta}(x_0,y_0,X,Y).
\end{align*}
Notice that in the above optimization problems, the maximum or minimum value needn't be attained. It that case, we use supremum or infimum instead.

This ranking can be computed by only one optimization problem.

\begin{theorem}\label{thmRobRank}
We have
\begin{subequations}\label{maxThmRobRankHalf}
\begin{align}
\delta^*=\max\ \delta \st 
&(1-\delta)y_0^Tu\geq1,\ (1+\delta)x_0^Tv \leq 1,\\ 
&(1+\delta)Yu-(1-\delta)Xv\leq0,\ u,v\geq0.
\end{align}
\end{subequations}
\end{theorem}

\begin{proof}
If \dmu is efficient for all data $(x'_0,y'_0,X',Y')\in\mna{O}_{\delta}(x_0,y_0,X,Y)$, then it is also efficient for $x'_0:=(1+\delta)x_0$, $y'_0:=(1-\delta)y_0$, $X':=(1-\delta)X$ and $Y':=(1+\delta)Y$. Hence the problem \nref{maxThmRobRankHalf} is feasible. Conversely, if \nref{maxThmRobRankHalf} is feasible, then \dmu is efficient for $x'_0:=(1+\delta)x_0$, $y'_0:=(1-\delta)y_0$, $X':=(1-\delta)X$ and $Y':=(1+\delta)Y$. Obviously, \dmu is efficient for any $x'_0\leq(1+\delta)x_0$, $y'_0\geq(1-\delta)y_0$, $X'\geq(1-\delta)X$ and $Y'\leq(1+\delta)Y$. Therefore, \dmu is efficient for all $\delta$-perturbation of data.

Suppose now that \dmu is inefficient. If \dmu becomes efficient for some $\delta'$-perturbation $(x'_0,y'_0,X',Y')\in\mna{O}_{\delta'}(x_0,y_0,X,Y)$, then $(1-\delta')x_0\leq x'_0$, $(1+\delta')y_0\geq y'_0$, $(1+\delta')X\geq X'$ and $(1-\delta')Y\leq Y'$, and hence $\delta:=-\delta'$ is feasible in  \nref{maxThmRobRankHalf}. This gives us the lower bound $\delta^*\geq-\delta'$.
If \dmu is efficient for no $\delta'$-perturbation $(x'_0,y'_0,X',Y')\in\mna{O}_{\delta'}(x_0,y_0,X,Y)$, then \nref{maxThmRobRankHalf} cannot be feasible for $\delta:=-\delta'$, whence $\delta^*\leq-\delta'$.
\end{proof}

The optimization problem \nref{maxThmRobRankHalf} is a nonlinear programming problem, but it can be solved effectively in polynomial time by a suitable  interior point method. This is because the problem belongs to the class of generalized linear fractional programming problems, which have the form of
\begin{align*}
\min\ \lambda \st
Ax\leq\lambda Bx,\ 
Cx\leq c,\ x\geq 0,
\end{align*}
where $Bx\geq0$ holds for all $x$ satisfying $Cx\leq c$, $x\geq 0$. Such problems are solvable in polynomial time; see \cite{FreJar1995, NesNem1995}.
This class is called generalized linear fractional linear programming because the problem can be equivalently expressed as
\begin{align*}
\min\ \left(\max_i \frac{(Ax)_i}{(Bx)_i}\right)  \st
Cx\leq c,\ x\geq 0.
\end{align*}

The classical DEA ranking is computed by means of linear programming by solving \nref{lpDEA}. Even though \nref{maxThmRobRankHalf} is effectively solvable, it is desirable to have a linear programming model, too, in order that the computing techniques and model class are the same. That is why we now focus on a linear programming approximation of \nref{maxThmRobRankHalf}.

Substituting $\tilde{u}:=u/(1-\delta)$ and $\tilde{v}:=v/(1+\delta)$ in \nref{maxThmRobRankHalf}, we obtain
\begin{subequations}\label{maxRobRankApprNonlin}
\begin{align}
\delta^*=\max\ \delta \st 
&y_0^T\tilde{u}\geq(1-\delta)^{-2},\ x_0^T\tilde{v} \leq (1+\delta)^{-2},\\ 
&Y\tilde{u}-X\tilde{v}\leq0,\ \tilde{u},\tilde{v}\geq0.
\end{align}
\end{subequations}
Using the linear approximation of the nonlinear terms around $\delta=0$ as follows
\begin{align*}
(1-\delta)^{-2}\approx 1+2\delta,\quad
(1+\delta)^{-2}\approx 1-2\delta,
\end{align*}
we arrive at the linear programming approximation
\begin{subequations}\label{maxRobRankApprDouble}
\begin{align}
\label{maxRobRankApprDoubleA}
\max\ \delta \st 
&y_0^T\tilde{u}\geq1+2\delta,\ x_0^T\tilde{v} \leq 1-2\delta,\\ 
\label{maxRobRankApprDoubleB}
&Y\tilde{u}-X\tilde{v}\leq0,\ \tilde{u},\tilde{v}\geq0.
\end{align}
\end{subequations}
Rescaling the coefficient $\delta$ by the factor $1/2$, we get
\begin{subequations}\label{maxRobRankApprHalf}
\begin{align}
\delta^*=\frac{1}{2}\max\ \delta \st 
&y_0^T\tilde{u}\geq1+\delta,\ x_0^T\tilde{v} \leq 1-\delta,\\ 
&Y\tilde{u}-X\tilde{v}\leq0,\ \tilde{u},\tilde{v}\geq0.
\end{align}
\end{subequations}

\paragraph{The resulting ranking.}
To avoid the division by 2 in \nref{maxRobRankAppr} and for better scalability, we propose the following DEA ranking. \dmu has the ranking $r=1+\delta^*$, where $\delta^*$ is the optimal solution of the linear program
\begin{subequations}\label{maxRobRankAppr}
\begin{align}
\label{maxRobRankApprA}
\delta^*=\max\ \delta \st 
&y_0^T\tilde{u}\geq1+\delta,\ x_0^T\tilde{v} \leq 1-\delta,\\ 
\label{maxRobRankApprB}
&Y\tilde{u}-X\tilde{v}\leq0,\ \tilde{u},\tilde{v}\geq0.
\end{align}
\end{subequations}
We will also consider for comparison the precise nonlinear model from Theorem~\ref{thmRobRank}
\begin{subequations}\label{maxThmRobRank}
\begin{align}
\label{maxThmRobRankA}
\delta^*=2\max\ \delta \st 
&(1-\delta)y_0^Tu\geq1,\ (1+\delta)x_0^Tv \leq 1,\\ 
\label{maxThmRobRankB}
&(1+\delta)Yu-(1-\delta)Xv\leq0,\ u,v\geq0.
\end{align}
\end{subequations}

\subsection{Fixed data}\label{ssRobApprFixed}

Sometimes it happens that certain part of data is inherently fixed and thus variation of all data is not meaningful (cf.\ Example~\ref{exABC}). Nevertheless, our approach is easily adapted for this case. The decision on efficiency of DMU's remains the same, only the value of the ranking will be changed -- increased for efficient DMU's and decreased for inefficient ones. For illustration suppose that the input data $X$ are fixed.

The model is directly adapted to this form, and the precise nonlinear model \nref{maxThmRobRank} is modified as
\begin{align*}
\delta^*=2\max\ \delta \st 
&(1-\delta)y_0^Tu\geq1,\ x_0^Tv\leq 1,\\ 
&(1+\delta)Yu-Xv\leq0,\ u,v\geq0.
\end{align*}
This model has again a form of  generalized linear fractional linear programming, and therefore efficiently solvable.

An approximate linear model can be derived as follows. Substitute $\tilde{u}:=u$ and $\tilde{v}:=v/(1+\delta)$ to obtain
\begin{align*}
2\max\ \delta \st 
&y_0^T\tilde{u}\geq(1-\delta)^{-1},\ 
x_0^T\tilde{v} \leq (1+\delta)^{-1},\\ 
&Y\tilde{u}-X\tilde{v}\leq0,\ \tilde{u},\tilde{v}\geq0.
\end{align*}
Linearizing nonlinear terms around zero, we get a linear model
\begin{align*}
2\max\ \delta \st 
&y_0^T\tilde{u}\geq 1+\delta,\ 
x_0^T\tilde{v} \leq 1-\delta,\\ 
&Y\tilde{u}-X\tilde{v}\leq0,\ \tilde{u},\tilde{v}\geq0.
\end{align*}
This is exactly the double of the optimal value of \nref{maxRobRankAppr}.

\section{Properties of the robust approach ranking}

The proposed ranking and the model \nref{maxRobRankAppr} have many interesting and desirable properties. Besides the fact that it is based on how DMU's are robustly stable in their (in)efficiency, we further have the following properties. 
Notice that as for the classical DEA ranking (the so called Units Invariance Theorem \cite{CooSei2007}), $r$ is invariant to scaling the units of input and output data.
Since $r$ can be larger than $1$, it is related to the notion of the so called super-efficiency pioneered by Andersen and Petersen \cite{AndPet1993}. Their approach is, however, not units invariant \cite{CooSei2007}.

\subsection{Basic properties}\label{ssRobBasic}

The following theorem states some basic properties of the proposed ranking based on \nref{maxRobRankAppr}. Notice that for the nonlinear model \nref{maxThmRobRank} we have $r\in[-1,3]$, but the conditions (2) and (3) of Proposition~\ref{propCompEffIneff} hold as well.

\begin{proposition}\label{propCompEffIneff}
We have 
\begin{enumerate}[(1)]\addtolength{\itemsep}{-0.3\baselineskip}
\item
$r\in[0,2]$, 
\item
$r\geq1$ if and only if \dmu is efficient,
\item
$r<1$ if and only if \dmu is inefficient.
\end{enumerate}
\end{proposition}

\begin{proof}\mbox{}
\begin{enumerate}[(1)]
\item
We have $r\geq0$ since $\delta=-1$, $\tilde{u}=0$ and $\tilde{v}=0$ is feasible for \nref{maxRobRankAppr}.
We have $r\leq2$ since $\delta>1$ cannot be feasible for \nref{maxRobRankAppr} due to the constraint $x_0^T\tilde{v} \leq 1-\delta$, which would turn to $x_0^T\tilde{v}<0$.
\item
If \dmu is efficient, then $\delta=0$ is feasible for \nref{maxRobRankAppr}, and hence $r\geq1$.

If \dmu is inefficient, then the linear system
\begin{align}\label{sysPfPropPropEff}
y_0^T{u}\geq1,\ x_0^T{v} \leq 1,\  
Y{u}-X{v}\leq0,\ {u},{v}\geq0
\end{align}
is infeasible. We claim that
\begin{align}\label{sysPfPropPropEffEps}
y_0^T{u}\geq1+\eps,\ x_0^T{v} \leq 1-\eps,\  
Y{u}-X{v}\leq0,\ {u},{v}\geq0
\end{align}
it remains infeasible for some sufficiently small $\eps>0$; this gives $r<1$. The claim is not hard to see since if it is not the case, then for any $\eps>0$ the system \nref{sysPfPropPropEff} has a solution $u_{\eps},v_{\eps}$. The convex polyhedron described by \nref{sysPfPropPropEffEps} is inclusion monotonic with respect to $\eps>0$. Thus, we can restrict to a bounded region, and as $\eps\to0$, the sequence of points $u_{\eps},v_{\eps}$ has an accumulation point $u,v$, which solves \nref{sysPfPropPropEff}; a contradiction.
\item
It follows from the previous point.
\qedhere
\end{enumerate}
\end{proof}

\subsection{Robustness interpretation}

We see that the novel ranking does not change (in)efficiency of the classical ranking. Moreover, we can use $r$ as the measure of inefficiency and efficiency. For efficiency in particular, such a measure was not much studied and the efficient DMU's were just ranked by the value of 1. However, the novel ranking $r$ is more delicate and says how much the efficient DMU's are close to inefficiency. More specifically, if $r=1+\delta^*\geq1$, then \dmu is efficient for any variation of the data up to $50\delta^*\%$ of their nominal values; moreover, all data coefficients may vary simultaneously and independently to each other.

\begin{example}\label{exABC}
Suppose that the outputs of three DMU's $A$, $B$ and $C$ are
$$
A:(2,4),\ \ 
B:(3,3),\ \ 
C:(4,2).
$$
\begin{center}
\psset{unit=4.8ex,arrowscale=1.5}
\begin{pspicture}(-0.4,-0.4)(5.8,5.5)
\newgray{mygray}{0.9}
\psaxes[ticksize=2pt,labels=all,ticks=all]
{->}(0,0)(-0.2,-0.2)(5.8,5)
\uput[-90](5.8,0){$y_1$}
\uput[180](0,5.){$y_2$}
\qdisk(2,4){2pt}\uput[45](2,4){$A$}
\qdisk(3,3){2pt}\uput[45](3,3){$B$}
\qdisk(4,2){2pt}\uput[45](4,2){$C$}
\end{pspicture}
\end{center}
According to the classical methods, all are considered as efficient with ranking 1. However, a slight change of the outputs of $B$ can change it and make $B$ inefficient. Our robust approach reflects it and ranks the DMU's by $1.1429$, $1$ and $1.1429$. This means that $B$ is on the border between efficiency and inefficiency, but still considered as efficient, while $A$ and $C$ are more stable in their efficiency. They remain efficient for any perturbation of the data values up to $7.14\%$.

In this example, however, the input data are fixed to be constantly one. Therefore, it is suitable to utilize the methods from Section~\ref{ssRobApprFixed}. Both the nonlinear model and its linear approximation yield the same rankings $1.2857$, $1$, $1.2857$, respectively. This says that the output data may simultaneously and independently vary within $\frac{1}{7}$ of their nominal values and $A,C$ remain efficient.
\end{example}

The above proposition also says that the ranking $r$ lies in the interval $[0,2]$. In contrast to the classical DEA ranking, the novel ranking is thus naturally normalized. In means that it can be used not only for comparing DMU's in one model, but for comparing DMU's from different, even unrelated models. Thus, $r$ gives as a universal ranking.

\begin{example}
Suppose that we have a ranking of banks like
$$
1.0062,\ 
0.986,\ 
1.0397,\ 
1.024,\ 
0.97263,\ 
1.0009,\ 
1.0438,\ 
0.96441,
$$
and suppose that we have a ranking of hospitals like
$$
1.21,\ 
0.65338,\ 
1.3254,\ 
0.6799,\ 
1.0382,\ 
0.60379,\ 
0.89957,\ 
1.2454.
$$
In the first case, all the banks have very similar ranking. Even though some are considered as efficient and some as inefficient, there is no substantial difference in their performance. On the other hand, performance of particular hospitals differs a lot. There are some considerably efficient and some highly inefficient. This example thus shows the universal feature of our approach to ranking. Notice that the classical DEA ranking cannot provide such an apparent conclusion.
\end{example}

\subsection{Order comparison}

\begin{proposition}\label{propCompApprox}
The DMU list sorted by \nref{maxRobRankAppr} is the same as the DMU list sorted by the classical DEA ranking.
\end{proposition}

\begin{proof}
Substituting $\hat{u}:=\tilde{u}/(1-\delta)$ and $\hat{v}:=\tilde{v}/(1-\delta)$, the problem \nref{maxRobRankAppr} yields
\begin{align*}
\max\ \delta \st 
&y_0^T\hat{u}\geq \frac{1+\delta}{1-\delta},\ x_0^T\hat{v} \leq 1,\\ 
&Y\hat{u}-X\hat{v}\leq0,\ \hat{u},\hat{v}\geq0.
\end{align*}
The classical DEA ranking can be formulated as
\begin{align*}
\max\ \alpha \st y_0^Tu\geq\alpha,\ x_0^Tv \leq 1,\ Yu-Xv\leq0,\ u,v\geq0.
\end{align*}
Since the function $({1+\delta})/({1-\delta})$ is increasing on $\delta\in[-1,1)$, both rankings are sorted in the same way.
\end{proof}

This observation shows the favourable property that new ranking does not reverse the order between DMU's based on the classical approach. The ranking values itself are of course different in general.
 
The following observation shows the same property for the original robust ranking approach based on the generalized linear fractional program \nref{maxThmRobRank}. Therefore the linear programming approximation \nref{maxRobRankAppr} is tight and preserves the order of rankings produced by \nref{maxRobRankAppr}.

\begin{proposition}\label{propCompOrig}
The DMU list sorted by \nref{maxThmRobRank} is the same as the DMU list sorted by the classical DEA ranking.
\end{proposition}

\begin{proof}
Rewrite \nref{maxThmRobRank} to \nref{maxRobRankApprNonlin}, and then use substitution $\hat{u}:=\tilde{u}(1+\delta)^2$ and $\hat{v}:=\tilde{v}(1+\delta)^2$. We get
\begin{align*}
\max\ \delta \st 
&y_0^T\hat{u}\geq \frac{(1+\delta)^2}{(1-\delta)^2},\ x_0^T\hat{v} \leq 1,\\ 
&Y\hat{u}-X\hat{v}\leq0,\ \hat{u},\hat{v}\geq0.
\end{align*}
Again, the function $({1+\delta})^2({1-\delta})^{-2}$ is increasing on $\delta\in[-1,1)$, so both rankings are sorted in the same way.
\end{proof}

\begin{proposition}
Denote by $r$, $r^\text{\nref{maxRobRankAppr}}$ and $r^\text{\nref{maxThmRobRank}}$ the classical CCR and the ranking computed by \nref{maxRobRankAppr} and \nref{maxThmRobRank}, respectively. 
\begin{enumerate}[(1)]\addtolength{\itemsep}{-0.3\baselineskip}
\item
If \dmu is efficient, then $r\leq r^\text{\nref{maxRobRankAppr}}\leq r^\text{\nref{maxThmRobRank}}$.
\item
Let \dmu be inefficient. Then $r^\text{\nref{maxThmRobRank}}\leq r^\text{\nref{maxRobRankAppr}}$ and $r\leq r^\text{\nref{maxRobRankAppr}}$. Moreover, $r\leq r^\text{\nref{maxThmRobRank}}$ provided $r^\text{\nref{maxThmRobRank}}\geq3-2\sqrt{2}\approx 0.1716$.
\end{enumerate}
\end{proposition}

\begin{proof}
Let \dmu be efficient. Then $r=1\leq r^\text{\nref{maxRobRankAppr}}$, so it remains to prove $r^\text{\nref{maxRobRankAppr}}\leq r^\text{\nref{maxThmRobRank}}$. Let $\tilde{u},\tilde{v},\delta$ be an optimal solution to \nref{maxRobRankApprDouble}. Define $u\coloneqq \frac{\alpha}{1+\delta} \tilde{u}$ and $v\coloneqq \frac{\alpha}{1-\delta} \tilde{v}$, where $\alpha>0$ is a parameter. Then $u,v,\delta$ satisfy \nref{maxThmRobRankB} and 
\begin{align*}
(1+\delta)y_0^Tu\geq\alpha(1+2\delta),\ (1-\delta)x_0^Tv \leq \alpha(1-2\delta).
\end{align*}
In order to fulfill $(1+\delta)x_0^Tv \leq 1$, we put $\alpha\coloneqq\frac{1-\delta}{(1+\delta)(1-2\delta)}$. The first of the above inequalities then reads
\begin{align*}
(1+\delta)y_0^Tu\geq \frac{(1-\delta)(1+2\delta)}{(1+\delta)(1-2\delta)}.
\end{align*}
In order that $(1-\delta)y_0^Tu\geq 1$, it must hold
\begin{align*}
\frac{(1-\delta)^2}{(1+\delta)^2}\,\frac{1+2\delta}{1-2\delta}\geq1.
\end{align*}
Simple manipulation shows that this is true whenever $\delta\geq0$, which is our case. Thus, $u,v,\delta$ is feasible for \nref{maxRobRankAppr}.

Let \dmu be inefficient. First we prove $r^\text{\nref{maxThmRobRank}}\leq r^\text{\nref{maxRobRankAppr}}$. Let ${u},{v},\delta$ be an optimal solution to \nref{maxThmRobRank}. Define $\tilde{u}\coloneqq \alpha(1+\delta) u$ and $\tilde{v}\coloneqq \alpha(1-\delta) {v}$, where $\alpha>0$ is a parameter. Then $\tilde{u},\tilde{v},\delta$ satisfy \nref{maxRobRankApprDoubleB} and
\begin{align*}
(1-\delta)y_0^T\tilde{u}\geq\alpha(1+\delta),\ 
(1+\delta)x_0^T\tilde{v}\leq\alpha(1-\delta).
\end{align*}
In order to fulfill $x_0^T\tilde{v} \leq 1-2\delta$, we put $\alpha\coloneqq\frac{(1-2\delta)(1+\delta)}{1-\delta}$. The first of the above inequalities then reads
\begin{align*}
(1-\delta)y_0^T\tilde{u}\geq \frac{(1-2\delta)(1+\delta)^2}{1-\delta}.
\end{align*}
In order that $y_0^T\tilde{u}\geq 1+2\delta$, it must hold
\begin{align*}
\frac{(1-2\delta)(1+\delta)^2}{(1-\delta)^2}\geq 1+2\delta.
\end{align*}
Manipulations as above show that this is true as long as $\delta\leq0$, which is our case. Thus, $\tilde{u},\tilde{v},\delta$ is feasible for \nref{maxRobRankApprDouble}.

Now we prove $r\leq r^\text{\nref{maxRobRankAppr}}$. Let ${u},{v}$ be an optimal solution to the CCR model \nref{lpDEA}, and define substitution $\delta\coloneqq (y_0^T{u}-1)(y_0^T{u}+1)^{-1}<0$, $\tilde{u}\coloneqq (1-\delta){u}$, $\tilde{v}\coloneqq (1-\delta){v}$. Then $\delta,\tilde{u},\tilde{v}$ satisfies \nref{maxRobRankApprB} and the second inequality in \nref{maxRobRankApprA}. By definition of $\delta$, we have $y_0^T{u}=(1+\delta)(1-\delta)^{-1}$, from which it follows that the first inequality in \nref{maxRobRankApprA} is satisfied and $y_0^T\tilde{u}= 1+\delta\leq r^\text{\nref{maxRobRankAppr}}$.

Eventually, we prove $r\leq r^\text{\nref{maxThmRobRank}}$.
Let $\tilde{u},\tilde{v}$ be an optimal solution to the CCR model \nref{lpDEA}, and define substitution ${u}\coloneqq \alpha\tilde{u}$, ${v}\coloneqq \beta\tilde{v}$, where $\alpha,\beta>0$ are parameters to be specified. In order that $v$ satisfies the second inequality in \nref{maxThmRobRankA}, we put $\beta\coloneqq(1+\delta)^{-1}$. In order that the first inequality in \nref{maxThmRobRankB} is a multiple of the inequality $Y\tilde{u}-X\tilde{v}\leq0$, we consider the equation $(1+\delta)\alpha=(1-\delta)\beta$, from which $\alpha=(1-\delta)(1+\delta)^{-2}$. In order that the first inequality in \nref{maxThmRobRankA} holds as equation, we have the constraint
$$
(1-\delta) y_0^T \alpha\tilde{u}=1,
$$
or
$$
(1-\delta)^2(1+\delta)^{-2} y_0^T \tilde{u}=1,
$$
giving raise to the value of $\delta<1$.
Now, $r\leq r^\text{\nref{maxThmRobRank}}$ holds true when $r=y_0^T \tilde{u}\leq 1+2\delta$ (herein, 2$\delta$ is the objective value of \nref{maxThmRobRank}), which equivalently reads
$$
(1-\delta)^{-2}(1+\delta)^{2} \leq 1+2\delta,
$$
from which $\delta\in[1-\sqrt{2},0]$. Therefore the condition holds true provided $\delta\geq1-\sqrt{2}$, or one of $r,r^\text{\nref{maxThmRobRank}}$ is at least $3-2\sqrt{2}$.
\end{proof}

The proposed robust approach ranking is suitable for diverse kinds of generalizations of the standard DEA model.

\subsection{Interval data}

Our approach is suitable to deal with interval data, too. Suppose we are given interval data $[\unum{x}_0,\onum{x}_0]$, $[\unum{y}_0,\onum{y}_0]$, $[\unum{X},\onum{X}]$ and $[\unum{Y},\onum{Y}]$, covering the not exactly known data $x_0$, $y_0$, $X$ and $Y$, respectively.
The traditional approach to deal with interval data is to determine the best case and worst case rankings \cite{DesSmi2002,EmrRos2012,EntMae2002,JabFia2004,KhaTav2015,MosSal2010}; the same idea is used in other interval-valued linear programming models \cite{Fie2006,Hla2012a,HlaCer2016a}. The best case happens in the setting $x_0:=\unum{x}_0$, $y_0:=\onum{y}_0$, $X:=\onum{X}$ and $Y:=\unum{Y}$, and the worst case happens for $x_0:=\onum{x}_0$, $y_0:=\unum{y}_0$, $X:=\unum{X}$ and $Y:=\onum{Y}$. Thus, it is sufficient to solve \nref{maxRobRankAppr} twice for both settings, and we get the efficiency range $[\unum{r},\onum{r}]$ of all possible rankings. The value of $\unum{r}\geq1$ means that \dmu is always efficient, whereas $\onum{r}<1$ means that \dmu is efficient for no possible realization of interval data.

\subsection{Additional DMU}

Let us now consider the situation when there is an additional DMU with input data vector $x_a$ and output data vector $y_a$, and the question now states: How the current ranking will change?
Obviously, the additional DMU means a novel constraint in the optimization formulation, so the ranking cannot increase; it can only remain the same or decrease.
For both the classical and novel rankings, the set of the additional data that are unfavourable in the sense that \dmu leaves off efficiency, has a nice geometrical form.

\begin{proposition}
Suppose \dmu is efficient. For both models \nref{lpDEA} and \nref{maxRobRankAppr}, the set of all values $x_a\in\R^{n_1}$ and $y_a\in\R^{n_1}$ causing \dmu to be inefficient forms a convex set the closure of which is a convex polyhedron.
\end{proposition}

\begin{proof}
On account of Proposition~\ref{propCompEffIneff}, it is sufficient prove it for the model \nref{lpDEA} only. \dmu is inefficient iff the linear system of inequalities
\begin{align}\label{sysPfPropAddDMU}
y_0^Tu\geq1 \ x_0^Tv \leq 1,\ Yu-Xv\leq0,\ u,v\geq0,\ 
\end{align}
is infeasible. Let $\mna{K}$ be the convex cone generated by the constraints 
$$
y_0^Tu\geq1 \ x_0^Tv \leq 1,\ Yu-Xv\leq0,\ u,v\geq0,
$$
that is, the smallest convex cone containing the points described by these constraints. Its closure is a convex polyhedral cone, but this cone itself needn't be closed since \nref{sysPfPropAddDMU} needn't be bounded. Then \dmu in the model with the additional DMU is inefficient iff the linear program
\begin{align*}
\min\ y_a^Tu-x_a^Tv \st (u^T,v^T)\in\mna{K}
\end{align*}
has the unique minimizer at the origin. Denoting $h_1,\dots,h_p$ the extremal directions of $\mna{K}$, we can reformulate it as 
\begin{align*}
(u^T,v^T)h_i<0\ \mbox{ or }\ (u^T,v^T)h_i\leq0, \quad i=1,\dots,p,
\end{align*}
where the strict inequality is when $h_i\in\mna{K}$ and the non-strict inequality is when $h_i$ lies in the closure of $\mna{K}$ only. This characterizes the convex set the closure of which is a convex polyhedral cone.
\end{proof}

Notice that linear programming problems depending on some parameters are thoroughly investigated in \cite{Gal1979,NozGud1974}.

\subsection{Additional inputs or outputs}

How will the ranking be affected by the situation when the decision maker gives an additional input or output \cite{Tak2000}?
Apparently, the efficiencies cannot worsen, so they remain the same or increase.
The set of the additional data that are favourable in the sense that an inefficient \dmu becomes efficient, has again a nice geometrical form.
Consider the situation one output vector $(y_b,y_c)$ is added.

\begin{proposition}
Suppose \dmu is inefficient. For both models \nref{lpDEA} and \nref{maxRobRankAppr}, the set of all values $(y_b,y_c)\in\R^{1+m}$ causing \dmu to be efficient forms a convex set the closure of which is a convex polyhedron.
\end{proposition}

\begin{proof}
\dmu is efficient in the extended model iff the linear system of inequalities
\begin{align*}
y_bu_0 + y_0^Tu\geq1, \ x_0^Tv \leq 1,\ y_cu_0+Yu-Xv\leq0,\ u,v\geq0
\end{align*}
is feasible. By the Farkas lemma, the dual system
\begin{align*}
-y_bz_1+y_c^Tz_3\geq0,\ 
-y_0z_1+Y^Tz_3\geq0,\ 
 x_0z_2-X^Tz_3\geq0,\ 
-z_1+z_2\leq-1,\ 
z_1,z_2,z_3\geq0
\end{align*}
is equivalently infeasible. Let $\mna{K}$ be the convex cone generated by the constraints 
\begin{align*}
-y_0z_1+Y^Tz_3\geq0,\ 
 x_0z_2-X^Tz_3\geq0,\ 
-z_1+z_2\leq-1,\ 
z_1,z_2,z_3\geq0.
\end{align*}
Then \dmu is efficient iff the linear program
\begin{align*}
\max\ -y_bz_1+y_c^Tz_3 \st (z_1,z_2,z_3^T)\in\mna{K}
\end{align*}
has the unique maximizer at the origin. Denoting $h_1,\dots,h_p$ the extremal directions of $\mna{K}$, we can reformulate it as 
\begin{align*}
(u^T,v^T)h_i<0\ \mbox{ or }\ (u^T,v^T)h_i\leq0, \quad i=1,\dots,p,
\end{align*}
where the strict inequality is when $h_i\in\mna{K}$ and the non-strict inequality is when $h_i$ lies in the closure of $\mna{K}$ only.
\end{proof}

\subsection{BCC model}\label{ssBCC}

Our robust approach to DMU ranking is suitable for employing various alternative models to CCR, and the evaluation technique is easily adapted. For illustration, consider the BCC model originally proposed by Banker et al.\ \cite{BanChar1984}; see also \cite{CooSei2007,Zhu2016}. The ranking of the test unit is computed by the linear program
\begin{align}\label{lpBCC}
\max\ y_0^Tu-v_0 \st x_0^Tv \leq 1,\ Yu-Xv-1v_0\leq0,\ u,v\geq0,
\end{align}
where $v_0$ is an additional free variable and $1$ is a vector of ones with convenient dimension.

Proceeding similarly as for the CCR model, we arrive at the following robust BCC model. The ranking of \dmu is $r=1+\delta^*$, where
\begin{subequations}\label{maxRobBcc}
\begin{align}
\delta^*=2\max\ \delta \st 
&(1-\delta)y_0^Tu-v_0\geq1,\ (1+\delta)x_0^Tv\leq 1,\\ 
&(1+\delta)Yu-(1-\delta)Xv-1v_0\leq0,\ u,v\geq0.
\end{align}
\end{subequations}
Again, this model belongs to the family of generalized linear fractional linear programming problems, and thus can be solved in polynomial time. Nevertheless, finding a linear programming approximation is still a useful issue.

Substitute $\tilde{u}:=u/(1-\delta)$, $\tilde{v}:=v/(1+\delta)$ and $\tilde{v}_0:=v_0/(1-\delta^2)$. We get
\begin{subequations}\label{maxRobBccTemp}
\begin{align}
2\max\ \delta \st 
&y_0^T\tilde{u}-\frac{1+\delta}{1-\delta}\tilde{v}_0\geq(1-\delta)^{-2},\ 
x_0^T\tilde{v} \leq (1+\delta)^{-2},\\ 
&Y\tilde{u}-X\tilde{v}-1\tilde{v}_0\leq0,\ \tilde{u},\tilde{v}\geq0.
\end{align}
\end{subequations}
Linearizing the nonlinear terms as follows
\begin{align*}
(1-\delta)^{-2}\approx 1+2\delta,\quad
(1+\delta)^{-2}\approx 1-2\delta,\quad
\frac{1+\delta}{1-\delta}\tilde{v}_0\approx \tilde{v}_0,
\end{align*}
and rescaling $\delta$, we arrive at a linear programming model
\begin{subequations}\label{maxRobBccLp}
\begin{align}
\delta^*=\max\ \delta \st 
&y_0^T\tilde{u}-\tilde{v}_0\geq 1+\delta,\ 
x_0^T\tilde{v} \leq 1-\delta,\\ 
&Y\tilde{u}-X\tilde{v}-1\tilde{v}_0\leq0,\ \tilde{u},\tilde{v}\geq0.
\end{align}
\end{subequations}

Now, we state some basic properties. Similarly as in Section~\ref{ssRobBasic}, we derive the following.

\begin{proposition}
For both the nonlinear model \nref{maxRobBcc} and linear model \nref{maxRobBccLp} we have 
\begin{enumerate}[(1)]\addtolength{\itemsep}{-0.3\baselineskip}
\item
$r\geq1$ if and only if \dmu is efficient,
\item
$r<1$ if and only if \dmu is inefficient.
\end{enumerate}
\end{proposition}


\begin{proposition}
For the nonlinear model \nref{maxRobBcc} we have $r\in[-1,3]$, and for the linear model \nref{maxRobBccLp} we have $r\in[0,2]$.
\end{proposition}

\begin{proposition}
The order of DMU's sorted by \nref{maxRobBcc} and \nref{maxRobBccLp} is the same as the DMU list sorted by the classical BCC model.
\end{proposition}

\begin{proof}
Starting from \nref{maxRobBccTemp}, the model equivalently reads
\begin{align*}
\max\ \delta \st 
&\frac{1}{(1+\delta)^2}y_0^T\hat{u}-\frac{1}{1-\delta^2}\hat{v}_0
  \geq \frac{1}{(1-\delta)^2},\ 
 x_0^T\hat{v} \leq 1,\\ 
&Y\hat{u}-X\hat{v}-1\hat{v}_0\leq0,\ \hat{u},\hat{v},\hat{v}_0\geq0,
\end{align*}
using substitution $\hat{u}:=\tilde{u}(1+\delta)^2$, $\hat{v}:=\tilde{v}(1+\delta)^2$ and $\hat{v}_0:=\tilde{v}_0(1+\delta)^2$. Rewrite the problem as
\begin{subequations}
\begin{align}
\label{lpPfThmBccSort}
\max\ \delta \st 
&\frac{1-\delta}{1+\delta}y_0^T\hat{u}-\hat{v}_0
  \geq \frac{1+\delta}{1-\delta},\ 
 x_0^T\hat{v} \leq 1,\\ 
&Y\hat{u}-X\hat{v}-1\hat{v}_0\leq0,\ \hat{u},\hat{v},\hat{v}_0\geq0.
\end{align}
\end{subequations}
Notice that the BCC model \nref{lpBCC} is equivalent to
\begin{align*}
\max\ \alpha \st 
 y_0^Tu-v_0 \geq \alpha,\ x_0^Tv \leq 1,\ Yu-Xv-1v_0\leq0,\ u,v\geq0.
\end{align*}
Now, the statement follows from the fact that function $(1+\delta)(1-\delta)^{-1}$ is increasing on $\delta\in[-1,1)$ and the function $(1-\delta)(1+\delta)^{-1}$ is decreasing on $\delta\in(-1,1]$. The former is the right-hand side in \nref{lpPfThmBccSort}, and the latter is the coefficient by the non-negative term $y_0^T\hat{u}$ therein.

For \nref{maxRobBccLp} the proof is analogous.
\end{proof}

\section{Examples}

Rather than a case study, we present several illustrative examples. They show the particular properties of the proposed ranking as discussed above.

\begin{example}\label{ex2}
Consider the example from \cite[Table 1.5]{CooSei2007}. The data displayed in Table~\ref{tabEx2} record two inputs (doctors and nurses) and two outputs (outpatients and inpatients) for 12 hospitals. The last three columns show the classical efficiency and the novel efficiency based on the robust approach by using the linear programming approximation or the generalized linear fractional programming formulation, respectively
\begin{table}[t]
\caption{(Example~\ref{ex2}) DEA with hospital data and resulting efficiencies.\label{tabEx2}}
\begin{center}
\small\footnotesize
\begin{tabular}{@{}cccccccc@{}}
\toprule
DMU & doctors & nurses & outpatients & inpatients & classical eff. & new eff. by \nref{maxRobRankAppr} & new eff. by \nref{maxThmRobRank}\\
\midrule
A& 20&151 & 100&90 & 1      & 1.1696 & 1.1708 \\
B& 19&131 & 150&50 & 1      & 1.0843 & 1.0845 \\
C& 25&160 & 160&55 & 0.8827 & 0.9377 & 0.9376 \\
D& 27&168 & 180&72 & 1      & 1.0079 & 1.0079 \\
E& 22&158 &\ 94&66 & 0.7635 & 0.8659 & 0.8653 \\
F& 55&255 & 230&90 & 0.8348 & 0.9100 & 0.9097 \\
G& 33&235 & 220&88 & 0.9020 & 0.9485 & 0.9484 \\
H& 31&206 & 152&80 & 0.7963 & 0.8866 & 0.8863 \\
I& 30&244 & 190&100& 0.9604 & 0.9798 & 0.9798 \\
J& 50&268 & 250&100& 0.8707 & 0.9309 & 0.9307 \\
K& 53&306 & 260&147& 0.9551 & 0.9770 & 0.9770 \\
L& 38&284 & 250&120& 0.9582 & 0.9787 & 0.9787 \\
\bottomrule
\end{tabular}
\end{center}
\end{table}

The computed results confirm that the linear programming model \nref{maxRobRankAppr} is indeed a very good approximation of \nref{maxThmRobRank} with the highest error $0.2\%$ and the average error $0.025\%$. Therefore we can utilize it for computing the efficiencies instead.

We also see satisfaction of the properties of Propositions~\ref{propCompApprox} and~\ref{propCompOrig} that the classical method yields the same order of efficiencies as our approach. Moreover, we can interpret the efficiencies based on their robustness aspect. Thus we can claim that hospital A is very efficient since it remains efficient even for rather large variations of data  -- all data can arbitrarily and independently perturb up to $8.54\%$ of their nominal values while preserving efficiency. In contrast, hospital D is only little efficient and close to inefficiency -- only $0.4\%$ variations are admissible and higher variations of data may result in inefficiency. From the other side, hospitals E and H are quite inefficient as they remain inefficient for any perturbation of data up to about $6\%$, but hospitals I, K and L are nearly efficient -- they can achieve efficiency by a suitable $1.15\%$-variation of data.
\end{example}

\begin{example}\label{exInt}
Consider the example from \cite{EntMae2002,HeXu2016} with one-dimensional input and two-dimensional interval output for 10 DMU's. The data as well as the computed results by \cite{HeXu2016}  and by our approach are given in Table~\ref{tabExInt}.
\begin{table}[t]
\caption{(Example~\ref{exInt}) DEA with interval data.\label{tabExInt}}
\begin{center}
\small\footnotesize
\begin{tabular}{cccccc}
\toprule
DMU & $X_1$ & $Y_1$ & $Y_2$ & efficiency by \cite{HeXu2016} & novel efficiency\\
\midrule
A&1& [0.8,\ 1.2] & [7.50,\ 8.50] & [1,\,1] & [1.0169,\ 1.1148] \\
B&1& [1.8,\ 2.2] & [2.50,\ 3.50] & [0.4222,\ 0.6227] & [0.5937,\ 0.7675]  \\
C&1& [1.6,\ 2.4] & [5.75,\ 6.25] & [0.7297,\ 0.9167] & [0.8437,\ 0.9566]  \\
D&1& [2.5,\ 3.5] & [2.75,\ 3.25] & [ 0.5247,\ 0.7809] & [0.6882,\ 0.8770]  \\
E&1& [2.8,\ 3.2] & [6.75,\ 7.25] & [0.9646,\ 1] & [0.9819,\ 1.1292]  \\
F&1& [3.8,\ 4.2] & [1.83,\ 2.17] & [0.6131,\ 0.7806] & [0.7601,\ 0.8768]  \\
G&1& [3.4,\ 4.6] & [4.50,\ 5.50] & [0.7940,\ 1] & [0.8852,\ 1.0643]  \\
H&1& [4.7,\ 5.3] & [1.50,\ 2.50] & [0.6984,\ 0.9635] & [0.8224,\ 0.9814]  \\
I&1& [5.6,\ 6.4] & [1.67,\ 2.33] & [0.8229,\ 1] & [0.9028,\ 1.0482]  \\
J&1& [6.7,\ 7.3] & [0.75,\ 1.25] & [1,\ 1] & [1.0229,\ 1.1318]  \\
\bottomrule
\end{tabular}
\end{center}
\end{table}

We see that DMU's A and J are efficient for each realization, whereas B, C, D, F, and H are inefficient for each realization; the others may or may not be efficient. This conclusion is the same as observed in \cite{HeXu2016}. Nevertheless, in comparison to \cite{HeXu2016}, informative value of our ranking is higher. First, we measure the degree of efficiency instead of ranking them by 1. Second, the ranking measures distance to (in)efficiency. Thus we see, for example, B or F are far to efficiency while H is possibly closer. We can also claim that E is either efficient or very close to efficiency for each realization.
\end{example}

\begin{example}\label{exBCC}
Consider Example 4.1 from \cite{CooSei2007} with one-dimensional input and output for 8 DMU's. The data as well as the computed results by the linear model \nref{maxRobBccLp} are displayed in Table~\ref{tabExBCC}.
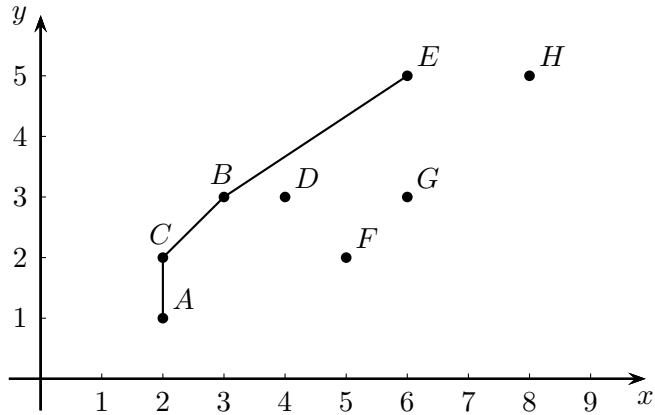
\begin{table}[t]
\caption{(Example~\ref{exBCC}) BCC model.\label{tabExBCC}}
\begin{center}
\small\footnotesize
\begin{tabular}{cccccc}
\toprule
DMU & input & output & classical BCC & robust BCC\\
\midrule
A & 2 & 1 & 1    & 1.00 \\
B & 3 & 3 & 1    & 1.05 \\
C & 2 & 2 & 1    & 1.11 \\
D & 4 & 3 & 0.75 & 0.86 \\
E & 6 & 5 & 1    & 1.14 \\
F & 5 & 2 & 0.40 & 0.57 \\
G & 6 & 3 & 0.50 & 0.67 \\
H & 8 & 5 & 0.75 & 0.86 \\
\bottomrule
\end{tabular}
\end{center}
\end{table}
\begin{center}
\psset{unit=4.85ex,arrowscale=1.5}
\begin{pspicture}(-0.5,-0.5)(11,6.5)
\newgray{mygray}{0.9}
\psaxes[ticksize=2pt,labels=all,ticks=all]
{->}(0,0)(-0.52,-0.52)(9.9,6)
\uput[-90](9.9,0){$x$}
\uput[180](0,6){$y$}
\qdisk(2,1){2pt}\uput[45](2,1){$A$}
\qdisk(3,3){2pt}\uput[95](3,3){$B$}
\qdisk(2,2){2pt}\uput[95](2,2){$C$}
\qdisk(4,3){2pt}\uput[45](4,3){$D$}
\qdisk(6,5){2pt}\uput[45](6,5){$E$}
\qdisk(5,2){2pt}\uput[45](5,2){$F$}
\qdisk(6,3){2pt}\uput[45](6,3){$G$}
\qdisk(8,5){2pt}\uput[45](8,5){$H$}
\psline[linewidth=0.8pt](2,1)(2,2)(3,3)(6,5)
\end{pspicture}
\end{center}
We see that the efficient units are the same for the classical BCC model and for our robust counterpart. Nevertheless, our model provides more information. We see directly that the unit A is not stable and arbitrarily small perturbation makes it inefficient. We also see that the units C and E are the most stable ones and they stay efficient even for more than $5\%$ variation of data. On the other hand, DMU F is the most inefficient and it remains inefficient for arbitrary perturbation of the data up to $21\%$.
\end{example}

\section{Conclusion}

We proposed a new DEA ranking that is based on robustness of DMU's of their (in)efficiency. This ranking has many attractive properties. As the classical DEA ranking, it is computable by means of linear programming, it is invariant with respect to scaling, and it gives a measure of efficiency as a distance to inefficiency and vice versa. In addition, the novel approach is naturally normalized, so it is suitable as a universal ranking technique of DMU's of possibly completely different models. It is also suitable for further generalization. We discussed models with interval data as an important extension of the classical real-valued problems, and also an adaptation to the BCC model.

We presented several examples that confirm interpretability of the novel efficiency ranking as well as applicability for model with real or interval data.
Attractivity of our approach is confirmed by an early application by Hol\'{y} \& \v{S}afr \cite{HolSaf2017}.

\subsubsection*{Acknowledgments.} 
The author wishes to thank the fruitful discussion at DEA 2017 conference in Prague.


\bibliographystyle{abbrv}
\bibliography{rob_dea}

\end{document}